\documentclass{amsart}
\usepackage{graphicx} 
\usepackage{cleveref}

\title{Benoist--Hulin groups}
\author{Hideki Miyachi and Yannian Zhao}
\date{\today}

\address{Hideki Miyachi and Yannian Zhao,
School of Mathematics and Physics,
College of Science and Engineering,
Kanazawa University,
Kakuma-machi, Kanazawa,
Ishikawa, 920-1192, Japan
}
\email{miyachi@se.kanazawa-u.ac.jp}
\email{yannianz@stu.kanazawa-u.ac.jp}

\subjclass[2020]{Primary 30C62; Secondary 57M60.}
\keywords{Quasidisks, quasiconformal mappings, Jordan curves.}

\thanks{The first author is supported by JSPS KAKENHI Grant Numbers 25K00909, 23K22396, 70183225.}

\newtheorem{theorem}{Theorem}[section]
\newtheorem{corollary}{Corollary}
\newtheorem{lemma}{Lemma}[section]

\begin{document}

\begin{abstract}
A Benoist--Hulin group is, by definition, a subgroup $\Gamma$ of ${\rm PSL}_2(\mathbb{C})$ such that any $\Gamma$-invariant closed set consisting of Jordan curves in the space of closed subsets of the Riemann sphere that are not singletons is composed of $K$-quasicircles for some $K \ge 1$. Y.~Benoist and D.~Hulin showed that the full group ${\rm PSL}_2(\mathbb{C})$ is a Benoist--Hulin group. In this paper, we develop the theory of Benoist--Hulin groups and show that both uniform lattices and parabolic subgroups are Benoist--Hulin groups.
\end{abstract}

\maketitle

\section{Introduction}
\subsection{Benoist--Hulin groups}
Let $\mathcal{K}$ denote the set of non-empty compact subsets of the Riemann sphere $\hat{\mathbb{C}}$, equipped with the Hausdorff distance. Let $\mathcal{K}_0 \subset \mathcal{K}$ be the subset consisting of those elements that are not singletons. The space $\mathcal{K}$ is a metric space, and $\mathcal{K}_0$ inherits the subspace topology from $\mathcal{K}$.  
The projective special linear group ${\rm PSL}_2(\mathbb{C})$ acts naturally on both $\mathcal{K}$ and $\mathcal{K}_0$.

Following Benoist and Hulin~\cite{BH}, a subgroup $\Gamma$ of ${\rm PSL}_2(\mathbb{C})$ is said to satisfy the \emph{Benoist--Hulin condition}, or the \emph{BH condition} for short, if every closed $\Gamma$-invariant subset of $\mathcal{K}_0$ consisting entirely of Jordan curves is contained in the set of $K$-quasicircles for some $K \geq 1$. A subgroup of ${\rm PSL}_2(\mathbb{C})$ that satisfies the BH condition is called a \emph{Benoist--Hulin group}, or a \emph{BH group} for short.
The BH condition is invariant under conjugation: if a subgroup $\Gamma$ of ${\rm PSL}_2(\mathbb{C})$ satisfies the BH condition, then so does $g\Gamma g^{-1}$ for every $g \in {\rm PSL}_2(\mathbb{C})$.

In~\cite{BH}, Benoist and Hulin observed that the full group ${\rm PSL}_2(\mathbb{C})$ is a BH group. They also remarked that BH groups are useful for characterizing quasicircles in the following sense: if $\Gamma$ is a BH group, then a Jordan curve $\gamma$ on $\hat{\mathbb{C}}$ is a quasicircle if and only if the closure of the $\Gamma$-orbit of $\gamma$ consists solely of points and Jordan curves.

Quasicircles have been extensively studied in function theory, and numerous characterizations of quasicircles are now known (cf. \cite{GH}). The Benoist--Hulin characterization is based on the dynamical properties of the group action on $\mathcal{K}_0$.

\subsection{Results}
The aim of this short paper is to develop a general theory for subgroups of ${\rm PSL}_2(\mathbb{C})$ satisfying the BH condition. We begin by proving the following:
\begin{theorem}
\label{thm:main3}
Any discrete BH group is of the first kind.
\end{theorem}
By definition, a BH group may be indiscrete. Nevertheless, \Cref{thm:main3} suggests that the action of a BH group on $\hat{\mathbb{C}}$ is highly nontrivial.
For elementary groups,
we will verify \Cref{thm:main3} by constructing a closed, $\Gamma$-invariant subset of $\mathcal{K}_0$ consisting of Jordan curves, where $\Gamma$ is a given elementary group (cf.~\S\ref{subsec:elementary}).

Since BH groups are defined in terms of the dynamical properties of their action on $\mathcal{K}_0$, it is natural to expect a form of \emph{stability}: if two subgroups of ${\rm PSL}_2(\mathbb{C})$ are sufficiently close and one satisfies the BH condition, then so does the other. Motivated by this intuition, we introduce the following notion.

Let $\Gamma_0 \subset \Gamma$ be subgroups of ${\rm PSL}_2(\mathbb{C})$. We say that $\Gamma_0$ is \emph{cocompact} in $\Gamma$ if there exists a compact subset $C \subset {\rm PSL}_2(\mathbb{C})$ such that
\[
\Gamma = C\Gamma_0=\bigcup_{\gamma \in \Gamma_0} C \gamma,
\]
where $AB := \{ a,b \mid a\in A, b \in B \}$ for some subsets $A$, $B\subset {\rm PSL}_2(\mathbb{C})$ and $C\gamma=C\{\gamma\}$.
With this definition, we obtain the following result:

\begin{theorem}
\label{thm:main}
Let $\Gamma$ and $\Gamma_0$ be subgroups of ${\rm PSL}_2(\mathbb{C})$ with $\Gamma_0 \subset \Gamma$. Suppose that $\Gamma_0$ is cocompact in $\Gamma$. Then $\Gamma$ satisfies the BH condition if and only if $\Gamma_0$ does.
\end{theorem}

Note that this notion of cocompactness is transitive: for subgroups $\Gamma_1 \supset \Gamma_2 \supset \Gamma_3$ of ${\rm PSL}_2(\mathbb{C})$, if $\Gamma_2$ is cocompact in $\Gamma_1$ and $\Gamma_3$ is cocompact in $\Gamma_2$, then $\Gamma_3$ is cocompact in $\Gamma_1$. 
When discussing the BH condition for subgroups, it is important to ensure that the notion under consideration respects transitivity.

As an immediate consequence of \Cref{thm:main}, we obtain the following:

\begin{corollary}
\label{coro:2}
Let $\Gamma$ and $\Gamma_0$ be subgroups of ${\rm PSL}_2(\mathbb{C})$. 
If $\Gamma_0$ is a finite-index subgroup of $\Gamma$, then $\Gamma$ satisfies the BH condition if and only if $\Gamma_0$ does.
\end{corollary}

\begin{corollary}
\label{coro:1}
Any uniform lattice in ${\rm PSL}_2(\mathbb{C})$ satisfies the BH condition. In particular, for a Jordan curve $\gamma$ on $\hat{\mathbb{C}}$, the following are equivalent:
\begin{itemize}
    \item[(a)] $\gamma$ is a quasicircle;
    \item[(b)] for any uniform lattice $\Gamma$ in ${\rm PSL}_2(\mathbb{C})$, the closure of the $\Gamma$-orbit of $\gamma$ in $\mathcal{K}$ consists only of points and Jordan curves;
    \item[(c)] for some uniform lattice $\Gamma$ in ${\rm PSL}_2(\mathbb{C})$, the closure of the $\Gamma$-orbit of $\gamma$ in $\mathcal{K}$ consists only of points and Jordan curves.
\end{itemize}
\end{corollary}

\Cref{coro:2} follows from the fact that any finite-index subgroup is cocompact. \Cref{coro:1} is a direct consequence of \Cref{thm:main} and the Benoist--Hulin characterization of quasicircles via the BH condition.

We also derive the following from \Cref{thm:main}:

\begin{corollary}
\label{coro:4}
A Borel subgroup of ${\rm PSL}_2(\mathbb{C})$ is a BH group. In particular, any parabolic subgroup of ${\rm PSL}_2(\mathbb{C})$ is a BH group.
\end{corollary}

A \emph{Borel subgroup} of an algebraic group $G$ is, by definition, a closed, connected, solvable subgroup that is not properly contained in any other such subgroup. A Borel subgroup is unique up to conjugacy. A \emph{parabolic subgroup} $P$ of $G$ is a closed subgroup such that the quotient $G/P$ is projective. A closed subgroup is parabolic if and only if it contains a Borel subgroup (cf.~\cite[Chapter~VIII, §21.3]{Hu}).
In the case of ${\rm PSL}_2(\mathbb{C})$, the subgroup $\mathcal{B}$ consisting of upper triangular elements of ${\rm PSL}_2(\mathbb{C})$ is a Borel subgroup.

\Cref{coro:4} follows from the Benoist--Hulin theorem \cite{BH} and from the fact, stemming from the Iwasawa decomposition, that
\[
{\rm PSL}_2(\mathbb{C}) = 
{\rm PSU}(2)\mathcal{B},
\]
where ${\rm SU}(2)$ denotes the $2 \times 2$ special unitary group (see also \Cref{lem:subgroup} in \S\ref{subsec:a_lemma}).

Since every element of the Borel subgroup fixes the point at infinity in $\hat{\mathbb{C}}$, any discrete subgroup of the Borel subgroup is elementary. Hence, the Borel subgroup contains no uniform lattice, and \Cref{coro:4} is independent of \Cref{coro:1}.

Furthermore, as discussed in \Cref{thm:main3}, the action of any BH group on $\hat{\mathbb{C}}$ exhibits rich dynamical behavior. Indeed, any \emph{discrete} BH group is Zariski dense. Nevertheless, \Cref{coro:4} shows that a BH group is not necessarily Zariski dense in ${\rm PSL}_2(\mathbb{C})$.

We conclude this section by posing a natural question arising from the above observations:
Is it true that a subgroup $\Gamma$ of ${\rm PSL}_2(\mathbb{C})$ is a BH group if and only if the set of fixed points of elements in $\Gamma$ is dense in $\hat{\mathbb{C}}$?




\subsection{About this paper}
This paper is organized as follows.  
In \S\ref{sec:subgroups}, we review notions from the theory of Kleinian groups.  
The proofs of the main results are given in \S\ref{sec:proofs}. Specifically, \Cref{thm:main3} is proved in \S\ref{subsec:1}, and \Cref{thm:main} in \S\ref{subsec:main}.  
We discuss the case of elementary Kleinian gruops in \S\ref{subsec:elementary}.

Part of the content of this paper appears in the second author's master's thesis.

\subsection*{Acknowledgements}
The authors thank Professor Kazuki Kannaka for suggesting \Cref{coro:4}.

\section{Notation}

\subsection{Quasicircles}
For references related to this section, see \cite{GH} for instance.

Let $D$ and $E$ be domains in $\hat{\mathbb{C}}$. An orientation-preserving homeomorphism $h\colon D \to E$ is said to be $K$-\emph{quasiconformal} for some $K\ge 1$ if the following two conditions are satisfied:
\begin{itemize}
\item[(i)] $h$ is absolutely continuous on lines; and
\item[(ii)] 
$|h_{\overline{z}}| \le \frac{K-1}{K+1} |h_z|$
almost everywhere on $D$.
\end{itemize}
It is known that$h$ is $1$-quasiconformal if and only if it is conformal. Especially, $1$-quasiconformal mapping on $\hat{\mathbb{C}}$ is a M\"obius transformation,and vice versa.

A \emph{Jordan curve} on $\hat{\mathbb{C}}$ is a simple closed curve, that is, a continuous injective map from the circle $S^1$ into $\hat{\mathbb{C}}$. A $K$-\emph{quasicircle} for some $K\ge 1$ is a Jordan curve which is the image of a round circle by a $K$-quasiconformal mapping on $\hat{\mathbb{C}}$. A \emph{quasicircle} is a $K$-quasicircle for some $K\ge 1$. The image of a $K$-quasicircle under a $K'$-quasiconformal mapping is a $KK'$-quasicircle.

\subsection{Kleinian groups}
\label{sec:subgroups}
For references related to this section, see \cite{Ka}, \cite{Ma}, \cite{MT}, and \cite{Oh}, for instance.

The group ${\rm PSL}_2(\mathbb{C})$ is identified with the group of biholomorphic automorphisms of the Riemann sphere $\hat{\mathbb{C}}$. The Riemann sphere can also be naturally identified with the ideal boundary of hyperbolic $3$-space $\mathbb{H}^3$, and the action of each element of ${\rm PSL}_2(\mathbb{C})$ extends to an isometry of $\mathbb{H}^3$.

A discrete subgroup of ${\rm PSL}_2(\mathbb{C})$ is called a \emph{Kleinian group}.  
The \emph{limit set} $\Lambda(\Gamma)$ of a Kleinian group $\Gamma$ is the set of accumulation points in $\hat{\mathbb{C}}$ of the orbit of a point in $\mathbb{H}^3$ under the action of $\Gamma$.

A Kleinian group $\Gamma$ is said to be \emph{elementary} if $\Lambda(\Gamma)$ is a finite set; otherwise, it is called \emph{non-elementary}.
It is known that the limit set of any elementary Kleinian group contains at most two points, and every elementary Kleinian group is virtually abelian.  
In fact, any non-trivial elementary Kleinian group contains as a finite-index subgroup either a cyclic group generated by a parabolic or loxodromic element, or a rank-$2$ parabolic subgroup.

A Kleinian group $\Gamma$ is said to be of the \emph{first kind} if $\Lambda(\Gamma) = \hat{\mathbb{C}}$, and of the \emph{second kind} otherwise (cf.~\cite[VII.F.6]{Ma}).
In particular, every elementary group is of the second kind.  
For a Kleinian group $\Gamma$ of the second kind, the complement $\Omega(\Gamma) = \hat{\mathbb{C}} \setminus \Lambda(\Gamma)$ is called the \emph{region of discontinuity}, on which $\Gamma$ acts properly discontinuously.
Moreover, the limit set $\Lambda(\Gamma)$ of a Kleinian group of the second kind has empty interior.

\section{Proofs of results}
\label{sec:proofs}

\subsection{A lemma}
\label{subsec:a_lemma}
We begin with the following elementary lemma.

\begin{lemma}
\label{lem:subgroup}
Let $\Gamma$ and $\Gamma_0$ be subgroups of ${\rm PSL}_2(\mathbb{C})$ with $\Gamma_0 \subset \Gamma$.
If $\Gamma_0$ is a BH group, then so is $\Gamma$.
\end{lemma}

\begin{proof}
Let $\mathcal{C} \subset \mathcal{K}_0$ be a $\Gamma$-invariant closed subset consisting of Jordan curves. Then $\mathcal{C}$ is also $\Gamma_0$-invariant. Since $\Gamma_0$ is a BH group, every element of $\mathcal{C}$ is a $K$-quasicircle for some $K \ge 1$.
\end{proof}

\subsection{Elementary Kleinian groups are not BH groups}
\label{subsec:elementary}

In this section, we verify that any elementary Kleinian group $\Gamma$ does not satisfy the BH condition. 
According to \Cref{coro:2}, it suffices to consider the following cases:
\begin{itemize}
\item[(i)] $\Gamma$ is the trivial group;
\item[(ii)] $\Gamma$ is a cyclic group generated by a parabolic element;
\item[(iii)] $\Gamma$ is a cyclic group generated by a loxodromic element; or
\item[(iv)] $\Gamma$ is an abelian group of rank two generated by two parabolic elements.
\end{itemize}
Note that \Cref{coro:2} is independent of the content of this section.

To this end, let $\gamma_{\mathrm{card}}$ be the boundary of the cardioid
$r = r(\theta) = 1 - \cos \theta$ in polar coordinates (cf.~\Cref{fig:Cardioid}).
Let $z(\theta) = (r(\theta)\cos \theta, r(\theta)\sin \theta)$ for $0 \le \theta \le 2\pi$ be its parametrization.
We first briefly verify that $\gamma_{\mathrm{card}}$ is not a quasicircle.
\begin{figure}[t]
\includegraphics[bb = 0 1 359 412, height = 3cm]{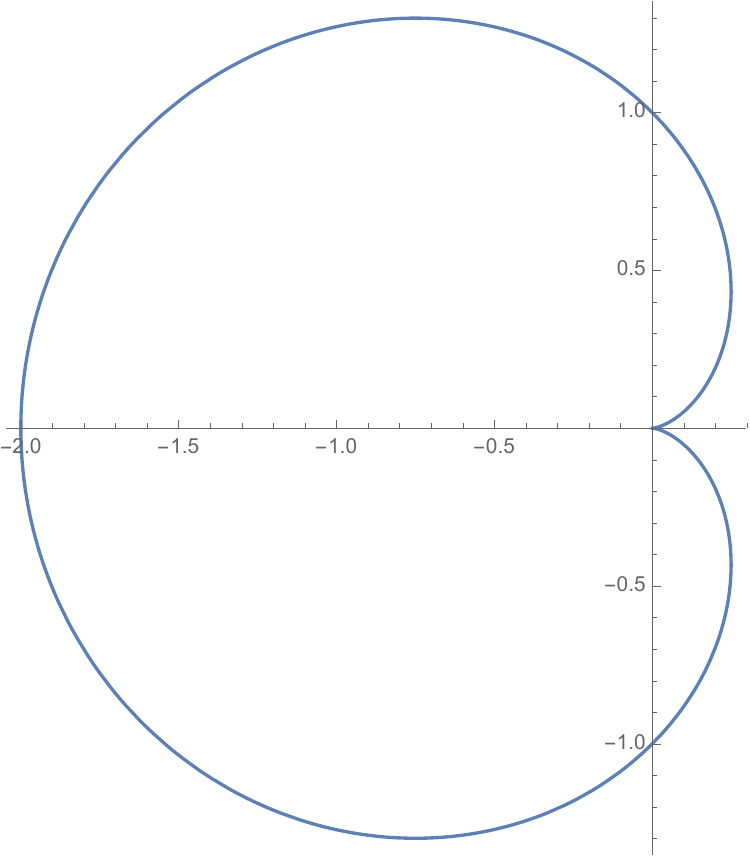}
\caption{Cardioid}
\label{fig:Cardioid}
\end{figure}

Indeed, near the cusp $z(0) = 0$, using the expansions 
$\cos \theta = 1 - \tfrac{\theta^{2}}{2} + O(\theta^{4})$ and 
$\sin \theta = \theta + O(\theta^{3})$ as $\theta \to 0$, we obtain 
$r(\theta) = 1 - \cos \theta = \tfrac{\theta^{2}}{2} + O(\theta^{4})$, and hence
\[
x(\theta) = r(\theta)\cos \theta = \tfrac{\theta^{2}}{2} + O(\theta^{4}), \qquad
y(\theta) = r(\theta)\sin \theta = \tfrac{\theta^{3}}{2} + O(\theta^{5}).
\] 
Thus, we have $|z(\pm \theta) - z(0)| \asymp \theta^{2}$ and 
$|z(\theta) - z(-\theta)| \asymp \theta^{3}$, so that
\[
\frac{|z(\theta) - z(0)|}{|z(\theta) - z(-\theta)|} \asymp \frac{1}{\theta} \to \infty
\]
as $\theta \to 0$.
This violates the Ahlfors three–point condition, and hence $\gamma_{\mathrm{card}}$ is not a quasicircle.

Let us now return to the verification that elementary Kleinian groups are not BH groups.
To this end, for each of the four cases listed above, we construct a closed, $\Gamma$–invariant subset of $\mathcal{K}_0$ consisting of Jordan curves whose quasicircle constants are unbounded.

\subsubsection*{Case (i)} $\Gamma_{\mathrm{id}}=\{\mathrm{id}\}$.  
Set $\mathcal C=\{\gamma_{\mathrm{card}}\}\subset\mathcal K_0$.  If $\Gamma_{\mathrm{id}}$ were BH, $\gamma_{\mathrm{card}}$ would be a $K$–quasicircle for some $K$, which is a contradiction.

\subsubsection*{Case (ii)} $\Gamma_{\mathrm{par}}=\langle T\rangle$, $T(z)=z+1$.  
Let
\[\mathcal C=\{\,T^{n}(\gamma_{\mathrm{card}}): n\in\mathbb Z\,\}\subset\mathcal K_0.\]
It is $\Gamma_{\mathrm{par}}$–invariant, consists of Jordan curves, and is closed in $\mathcal K_0$ (any Hausdorff limit of distinct translates escapes to $\{\infty\}$).
Since each element is a translate of $\gamma_{\mathrm{card}}$, no finite $K$ bounds their
quasicircle constants. Hence $\Gamma_{\mathrm{par}}$ does not satisfy the BH condition.

\subsubsection*{Case (iii)} $\Gamma_{\mathrm{lox}}=\langle L\rangle$, $L(z)=\lambda (z+1)+1$ ($|\lambda|>1$).
Define
\[\mathcal C=\{\,L^{n}(\gamma_{\mathrm{card}}): n\in\mathbb Z\,\}\subset\mathcal K_0.\]
Again $\mathcal C$ is $\Gamma_{\mathrm{lox}}$–invariant, closed (limits go to either $\{0\}$ or $\{\infty\}$), and every member is a scaled copy of $\gamma_{\mathrm{card}}$, so $\sup_{\gamma\in\mathcal C}K(\gamma)=\infty$. Thus $\Gamma_{\mathrm{lox}}$ does not satisfy the BH condition.

\subsubsection*{Case (iv)} $\Gamma_{\mathrm{rk2}}=\langle T_1,T_2\rangle$ with $T_1(z)=z+1$ and $T_2(z)=z+\tau$ ($\Im\tau>0$).  
Let\[\mathcal C=\{\,T_1^n\circ T_2^m(\gamma_{\mathrm{card}}) : m,n\in\mathbb Z\,\}\subset\mathcal K_0.\]
This is closed (limits escape to $\{\infty\}$), $\Gamma_{\mathrm{rk2}}$–invariant, and all elements are translates of $\gamma_c$. Again $\sup_{\gamma\in\mathcal C}K(\gamma)=\infty$, contradicting BH. Therefore $\Gamma_{\mathrm{rk2}}$ does not satisfy the BH condition.

\medskip
In all four cases we reach a contradiction, so every elementary Kleinian group
fails the BH condition.

\subsection{Proof of \Cref{thm:main3}}
\label{subsec:1}
We suppose that $\Gamma$ is non-elementary. The case of elementary groups is already discussed in \S\ref{subsec:elementary}.

Let $\Gamma$ be a Kleinian group of the second kind. Since the set of fixed points of elliptic elements in $\Gamma$ is discrete in $\Omega(\Gamma)$, there exists a relatively compact simply connected domain $D \subset \Omega(\Gamma)$ such that $\gamma(D) \cap D = \emptyset$ for all $\gamma \in \Gamma \setminus \{ \mathrm{id} \}$.

Let $\alpha$ be a Jordan curve contained in $D$ that is not a quasicircle. For instance, one may take $\alpha$ to be the boundary of the image of a cardioid under an affine map.
Let $\mathcal{C}$ be the closure in $\mathcal{K}_0$ of the $\Gamma$-orbit of $\alpha$. By construction, $\mathcal{C}$ contains a non-quasicircle. To show that $\Gamma$ is not a BH group, it suffices to prove that $\mathcal{C}$ consists only of Jordan curves.

Let $\beta \in \mathcal{C}$. Then one of the following holds:
\begin{itemize}
\item[(a)] $\beta = \gamma(\alpha)$ for some $\gamma \in \Gamma$; or
\item[(b)] there exists a sequence $\{ \gamma_n \} \subset \Gamma$ such that $\gamma_n(\alpha) \to \beta$ in $\mathcal{K}_0$.
\end{itemize}

In case (a), since $\alpha$ is a Jordan curve and Möbius transformations preserve this property, $\beta$ is also a Jordan curve.

Now consider case (b). Let $g_n = \gamma_n|_D$. Since $\Gamma$ is non-elementary, the limit set contains at least three points. Since $D\subset \Omega(\Gamma)$, $g_n(D)=\gamma_n(D)\subset\Omega(\Gamma)$ for $n\ge 1$, from the Montel theorem, the sequence $\{ g_n \}_{n=1}^\infty$ is a normal family. By passing to a subsequence if necessary, we may assume $g_n \to g_\infty$ uniformly on compact subsets of $D$, where $g_\infty$ is holomorphic on $D$.

Since $\Gamma$ acts properly discontinuously on $\Omega(\Gamma)$, the images $g_n(D)$ eventually exit any compact subset of $\Omega(\Gamma)$. Hence the image $g_\infty(D)$ is contained in the limit set $\Lambda(\Gamma)$. As $\Gamma$ is of the second kind, $\Lambda(\Gamma)$ has empty interior, so the open mapping theorem implies $g_\infty$ must be constant.

Therefore, $\gamma_n(\alpha) = g_n(\alpha) \to g_\infty(\alpha)$, which is a point. Thus $\beta$ is a singleton, not a Jordan curve, contradicting the assumption that $\beta \in \mathcal{C} \subset \mathcal{K}_0$. We conclude that case (b) cannot occur.

Hence, $\mathcal{C} = \{ \gamma(\alpha) \mid \gamma \in \Gamma \}$ and consists entirely of Jordan curves.

\subsection{Proof of \Cref{thm:main}}
\label{subsec:main}

By \Cref{lem:subgroup}, if $\Gamma_0$ satisfies the BH condition, then so does $\Gamma$.

Now suppose $\Gamma$ satisfies the BH condition. Let $\mathcal{C} \subset \mathcal{K}_0$ be a $\Gamma_0$-invariant closed subset consisting of Jordan curves, and define
\[
\hat{\mathcal{C}} = \{ \eta(\alpha) \in \mathcal{K}_0 \mid \eta \in C,\ \alpha \in \mathcal{C} \}.
\]
Since each $\eta \in C$ is a Möbius transformation, and Möbius transformations preserve Jordan curves, it follows that $\hat{\mathcal{C}}$ also consists of Jordan curves.

We claim that $\hat{\mathcal{C}}$ is a closed, $\Gamma$-invariant subset of $\mathcal{K}_0$.

Let $\alpha \in \hat{\mathcal{C}}$ and $g \in \Gamma$. Then $\alpha = \eta(\beta)$ for some $\eta \in C$ and $\beta \in \mathcal{C}$. Since $\Gamma = \bigcup_{\gamma \in \Gamma_0} C \gamma$ and $C \cdot \mathrm{id} \subset \Gamma$, there exist $\eta_1 \in C$ and $\gamma_1 \in \Gamma_0$ such that $g \eta = \eta_1 \gamma_1$. Since $\mathcal{C}$ is $\Gamma_0$-invariant, $\gamma_1(\beta) \in \mathcal{C}$, and hence
\[
g(\alpha) = g(\eta(\beta)) = \eta_1(\gamma_1(\beta)) \in \hat{\mathcal{C}},
\]
showing that $\hat{\mathcal{C}}$ is $\Gamma$-invariant.

To show that $\hat{\mathcal{C}}$ is closed, let $\alpha_n = \eta_n(\beta_n) \in \hat{\mathcal{C}}$ converge to $\alpha_0$ in $\mathcal{K}_0$. Since $C$ is compact, we may assume $\eta_n \to \eta_0 \in C$. Then
\[
\beta_n = \eta_n^{-1}(\alpha_n) \to \eta_0^{-1}(\alpha_0) =: \beta_0.
\]
As each $\beta_n \in \mathcal{C}$ and $\mathcal{C}$ is closed, we conclude that $\beta_0 \in \mathcal{C}$ and $\alpha_0 = \eta_0(\beta_0) \in \hat{\mathcal{C}}$. Thus, $\hat{\mathcal{C}}$ is a closed set in $\mathcal{K}_0$.

Since $\Gamma$ satisfies the BH condition, there exists $K \ge 1$ such that every element of $\hat{\mathcal{C}}$ is a $K$-quasicircle. This means that for any $\beta \in \mathcal{C}$ and any $\eta \in C$, $\eta(\beta)$ is a $K$-quasicircle. Since Möbius transformations preserve the class of $K$-quasicircles, $\beta$ must also be a $K$-quasicircle. Therefore, $\Gamma_0$ satisfies the BH condition.


\begin{thebibliography}{99}
\bibitem[BH]{BH} 
Y.~Benoist and D.~Hulin,
\newblock Quasicircles and the conformal group,
\newblock Conform. Geom. Dyn. {\bf 20} (2016), 282--302.
\bibitem[GH]{GH}
F.~W.~Gehring and K.~Hag,
\newblock {\it The ubiquitous quasidisk}, 
\newblock Mathematical Surveys and Monographs, {\bf 184}, Amer. Math. Soc., Providence, RI, 2012.
\bibitem[Hu]{Hu}
J.~E. Humphreys, 
\newblock {\it Linear algebraic groups},
\newblock Graduate Texts in Mathematics, No. 21, Springer, New York-Heidelberg, 1975.
\bibitem[Ka]{Ka}
M.~Kapovich,
\newblock {\it Hyperbolic manifolds and discrete groups}, 
\newblock
Progress in Mathematics, {\bf 183}, Birkh\"auser Boston, Boston, MA, 2001.
\bibitem[Ma]{Ma}
B.~Maskit,
\newblock {\it Kleinian groups},
\newblock Grundlehren der mathematischen Wissenschaften, {\bf 287}, Springer, Berlin, 1988.
\bibitem[MT]{MT}
K.~Matsuzaki and M.~Taniguchi,
\newblock {\it Hyperbolic manifolds and Kleinian groups}, 
\newblock Oxford Mathematical Monographs Oxford Science Publications, , Oxford Univ. Press, New York, 1998.
\bibitem[Oh]{Oh}
K.~Ohshika, 
\newblock {\it Discrete groups}, 
\newblock translated from the 1998 Japanese original by the author, 
Translations of Mathematical Monographs Iwanami Series in Modern Mathematics, {\bf 207} , Amer. Math. Soc., Providence, RI, 2002.
\end{thebibliography}
\end{document}